\documentclass[12pt]{amsart}
\usepackage{amssymb,amsmath,amsfonts,latexsym, mathtools}
\usepackage{bm, enumerate}
\usepackage{graphicx}
\usepackage{hyperref}
\newcommand{\newsection}[1]{\setcounter{equation}{0} \section{#1}}
\newcommand{\bea}{\begin{eqnarray}}
\newcommand{\eea}{\end{eqnarray}}

\newcommand{\clb}{\mathcal{B}}

\newcommand{\cld}{\mathcal{D}}
\newcommand{\cle}{\mathcal{E}}
\newcommand{\clf}{\mathcal{F}}

\newcommand{\clh}{\mathcal{H}}
\newcommand{\clk}{\mathcal{K}}

\newcommand{\clo}{\mathcal{O}}

\newcommand{\clq}{\mathcal{Q}}
\newcommand{\clr}{\mathcal{R}}

\newcommand{\clt}{\mathcal{T}}

\newcommand{\D}{\mathbb{D}}
\newcommand{\N}{\mathbb{N}}
\newcommand{\C}{\mathbb{C}}
\newcommand{\bs}{\mathbb S}

\newcommand{\raro}{\rightarrow}

\def \qed {\hfill \vrule height6pt width 6pt depth 0pt}
\def\textmatrix#1&#2\\#3&#4\\{\bigl({#1 \atop #3}\ {#2 \atop #4}\bigr)}
\def\dispmatrix#1&#2\\#3&#4\\{\left({#1 \atop #3}\ {#2 \atop #4}\right)}
\newcommand{\be}{\begin{equation}}
\newcommand{\ee}{\end{equation}}
\newcommand{\ben}{\begin{eqnarray*}}
\newcommand{\een}{\end{eqnarray*}}

\newcommand{\bi}{\begin{itemize}}
\newcommand{\ei}{\end{itemize}}

\newtheorem{Theorem}{\sc Theorem}[section]
\newtheorem{Lemma}[Theorem]{\sc Lemma}
\newtheorem{Proposition}[Theorem]{\sc Proposition}
\newtheorem{Corollary}[Theorem]{\sc Corollary}
\newtheorem{Definition}[Theorem]{\sc Definition}
\newtheorem{Example}[Theorem]{\sc Example}
\newtheorem{Remark}[Theorem]{\sc Remark}

\newtheorem{Note}[Theorem]{\sc Note}
\newtheorem{Question}{\sc Question}
\newtheorem{ass}[Theorem]{\sc Assumption}
\newcommand{\bt}{\begin{Theorem}}
\def\beginlem{\begin{Lemma}}
\def\beginprop{\begin{Proposition}}
\def\begincor{\begin{Corollary}}
\def\begindef{\begin{Definition}}
\def\beginexamp{\begin{Example}}
\def\beginrem{\begin{Remark}}
\def\beginq{\begin{Question}}
\def\beginass{\begin{ass}}
\def\beginnote{\begin{Note}}
\newcommand{\et}{\end{Theorem}}
\def\endlem{\end{Lemma}}
\def\endprop{\end{Proposition}}
\def\endcor{\end{Corollary}}
\def\enddef{\end{Definition}}
\def\endexamp{\end{Example}}
\def\endrem{\end{Remark}}
\def\endq{\end{Question}}
\def\endass{\end{ass}}
\def\endnote{\end{Note}}

\begin{document}

\title{Factors of hypercontractions}

\author[Bhattacharjee]{Monojit Bhattacharjee}
\address{Department of Mathematics, Indian Institute of Technology Bombay, Powai, Mumbai, 400076,
India}
\email{mono@math.iitb.ac.in, monojit.hcu@gmail.com}

\author[Das] {B. Krishna Das}
\address{Department of Mathematics, Indian Institute of Technology Bombay, Powai, Mumbai, 400076,
India}
\email{dasb@math.iitb.ac.in, bata436@gmail.com}

\subjclass[2010]{47A13, 47A20, 47A45, 47A56, 46E22, 47B32, 32A36, 47B20} \keywords{Hypercontraction, Bergman space over the unit disc, commuting contractions,  bounded analytic functions}

\begin{abstract}
In this article, we study a class of contractive factors
of $m$-hypercontractions for  $m\in \N$.
We find a characterization of such factors 
and this is achieved by finding explicit dilation of these factors on 
some weighted Bergman spaces. This is a generalization of the work 
done in ~\cite{DSS}.
\end{abstract}

\maketitle

\newsection{Introduction}
The structure of a commuting $n$-tuple of isometries $(n\ge 2)$ 
is complicated compare to that of a single isometry due to 
von Neumann and
Wold (cf. \cite{NF}).
Not much is known except the BCL representation 
for an $n$-tuple of isometries with product being a pure isometry
 (see 
 \cite{ BDF1, BDF2,BDF,BCL,DDS, GY, Yang} and references therein), 
 that is for an $n$-tuple of isometries $(V_1,\dots,V_n)$ on $\clh$
 with  
\[
 \cap_{k\ge 0}V_1^kV_2^k\cdots V_n^k \clh =\{0\}.
\]
The structure theorem of such isometries 
also reveals all possible isometric factors of a pure isometry ~\cite{BCL}.
Following this, the analysis of finding factors has been extended 
further to the case of contractions, recently. 
A characterization of contractive factors of a pure contraction
is obtained, by Sarkar, Sarkar and the second author of this atricle,
 in ~\cite{DSS} and subsequently in ~\cite{Hari} for 
general contractions.
More specifically, it is shown that for a contraction 
$T$ on a Hilbert space $\clh$, the following are equivalent: 
\begin{enumerate}[(i)]
\item $T=T_1T_2$ for some commuting
contractions $T_1$ and $T_2$ on $\clh$; \\
\item there exist a triple $(\cle,U,P)$
consisting of a Hilbert space 
$\cle$, a unitary $U$ and an orthogonal projection $P$,
 a pair of commuting unitaries $(W_1, W_2)$ 
 on a Hilbert space $\clr$ with $W=W_1W_2$ and a joint
$(M_z^*\oplus W^*, M_{\Phi}^* \oplus W_1^*, M_{\Psi}^* \oplus W_2^*)$-invariant subspace $\clq$ of $H^2_{\cle}(\D) \oplus \clr $ such that 
\[ 
T_1 \cong P_{\clq}(M_{\Phi} \oplus W_1)|_{\clq},
\,\, T_2 \cong P_{\clq}(M_{\Psi} \oplus W_2)|_{\clq},
\,\, T \cong P_{\clq}(M_z \oplus W)|_{\clq}
\] 
where $\Phi(z) = (P+zP^{\perp})U^*$
and $\Psi(z)= U(P^{\perp} + zP)$ for all $z $ in the unit disc $ \D$.   
\end{enumerate}
Moreover, $(M_{\Phi} \oplus W_1)(M_{\Psi} \oplus W_2)
= (M_{\Psi} \oplus W_2)(M_{\Phi} \oplus W_1)
= M_z \oplus W$.    
In the case of a pure contraction $T$,
the Hilbert space $\clr=\{0\}$
and therefore all the direct summands disappear.
It is also worth mentioning here that the key to obtain
such a characterization is an explicit
Ando type dilation result and it is motivated
by a recent technique of finding explicit dilation found in ~\cite{DS}.
It is then natural to ask the following question:
How to characterize contractive factors of $m$-hypercontractions?
In this article, we 
answer this question and obtained a complete description
 for a class of contractive factors of $m$-hypercontractions.
 Our characterization for contractive factors of $m$-hypercontractions 
 induces a similar characterization of factors for subnormal operators 
 and, for $m=1$, it recovers the characterization obtained in ~\cite{DSS}
 and ~\cite{Hari}.
To describe these results, we develop
some background materials next.



For a Hilbert space $\cle$ and $n\in \N$,
the $\cle$-valued weighted Bergman space over the unit disc,
denoted by $A^2_n(\cle)$, is defined as 
\[ 
A^2_n(\cle)= \{f \in \clo(\D, \cle): 
f(z)= \sum_{k=0}^{\infty} \hat{f}(k)z^k,\|f\|^2_n = \sum_{k=0}^{\infty} (w_{n,k})^{-1} \|\hat{f}(k)\|^2_{\cle} < \infty \}, 
\] 
where the sequence of weights $\{w_{n,k}\}_{k\ge 0}$ is given by
\[ 
(1-x)^{-n} = \sum_{k=0}^{\infty} w_{n,k} x^k, \hspace{1cm} (|x| < 1).
\]
It is also a reproducing kernel Hilbert space with kernel 
\[
 K_n(z,w)=(1-z\bar{w})^{-n}I_{\cle}\quad (z,w\in\D).
\]
For the base case $n=1$, the space $A^2_1(\cle)$
is known as the Hardy space over the unit disc
which we denote by $H^2_{\cle}(\D)$ and denote
the corresponding kernel, known as the Szeg\"{o} kernel, by 
\[
\mathbb S(z,w)=(1-z\bar{w})^{-1}I_{\cle}\quad (z,w\in\D).
\]   
If $\cle=\C$, then we denote simply by $A^2_n$ the 
$\C$-valued weighted Bergman space over the unit disc.
The notion of $m$-hypercontractions $(m\in\N)$,
introduced by Agler in his seminal paper ~\cite{JA2},
is defined as follows.
A bounded linear operator $T$ on $\clh$ is an
$m$-hypercontraction if it satisfies 
\begin{align*}
K_n^{-1}(T,T^*)  
     &= \sum_{k=0}^n (-1)^k \left(
    \begin{array}{c}
      n \\
      k
    \end{array} \right)T^kT^{*k} \geq 0,
  \end{align*}  
for $n=1, m$. In addition,
if $T^{* n}\to 0$ in the strong operator topology then $T$ is
said to be a pure $m$-hypercontraction.
It is important to note that the positivity $K_n^{-1}(T,T^*)\ge 0$ 
for $n=1,m$ also implies all the intermediate positivity, that is 
$K_n^{-1}(T,T^*)\ge 0$ for all $n=1,\dots, m$ (\cite{MV}).
This shows that if $T$ is an $m$-hypercontraction then it is also 
an $n$-hypercontraction for all $n=1,\dots,m$.
The \textit{defect operators} and the \textit{defect spaces} of an $m$-hypercontraction $T$ on $\clh$ are defined by 
\begin{align}
\label{defect}
D_{n,T}= \Big(K_n^{-1}(T,T^*) \Big)^{\frac{1}{2}}\ \text{ and }
\cld_{n,T}= \overline{\text{ran}}D_{n,T}, \quad (1\le n\le m) 
\end{align} respectively.
The Bergman shift $M_z$ on $A^2_m(\cle)$,  defined by
\[ \big(M_zf\big)(w)= wf(w) \quad (f \in A^2_m(\cle), w\in \D), \]
is a pure $m$-hypercontraction.
In fact, by ~\cite{JA2}, 
the Bergman shifts are model of
pure $m$-hypercontractions. 
To be more precise, Agler proves the following characterization result.

\begin{Theorem}\textup{(cf. ~\cite{JA2})}\label{Agler}
If $T$ is an $m$-hypercontraction on a Hilbert space $\clh$
then 
\[T\cong P_{\clq}(M_z\oplus W)|_{\clq},\]
 where
$W$ is a unitary on a Hilbert space $\clr$, 
$\clq$ is a $(M_z^*\oplus W^*)$-invariant subspace of
$A^2_m(\cld_{m,T})\oplus \clr$ and
$\cld_{m,T}$ is the defect space of $T$ as in ~\eqref{defect}.
In addition, if $T$ is pure then the Hilbert space 
$\clr=\{0\}$.
\end{Theorem}
\noindent There are now several different approach to this result 
and to its multivariable generalization
for different domains in $\C^n$
(see ~\cite{JA1},\cite{AM}, \cite{AEM}, ~\cite{CV1},
~\cite{CV2} , \cite{MV} and ~\cite{O}).


Now coming back to the context of this article,
we denote by $\clf_m(\clh)$ the class of contractive factors
of  $m$-hypercontractions on a Hilbert space $\clh$
which we characterize in this paper. The class 
is defined as follows.
\begin{Definition}
For $m\in\N$ and a Hilbert space $\clh$, a pair of operators 
$(T_1,T_2)$ is said to be an element of $\clf_m(\clh)$
if 
\begin{itemize}
 \item [\textup{(i)}]
 $T_1$ and $T_2$ are commuting contractions, and 
 \item [\textup{(ii)}]
 for all $i=1,2$, $K_{m-1}^{-1}(T,T^*)-T_iK_{m-1}^{-1}(T,T^*)T_i^*\ge 0$
 where $T=T_1T_2$ and $K_0(T,T^*)=I_{\clh}$. 
\end{itemize}
\end{Definition}
The positivity condition in (ii) is equivalent to the Szeg\"{o}
positivity of the commuting operator tuple
\[
\mathcal T_i=(\underbrace{T,\dots, T}_{(m-1)-\text{times}},T_i)
\]
for all $i=1,2$. Here for an $n$-tuple of commuting contraction 
$\mathcal T=(T_1,\dots, T_n)$, the Szeg\"{o} positivity of $\clt$
is defined as 
\[
 \bs_n^{-1}(\clt, \clt^*)= 
 \sum_{F\subset \{1,\dots,n\}}(-1)^{|F|}\clt_F\clt_F^*,
\]
where for $F\subset \{1,\dots,n\}$, $\clt_F=\prod_{i\in F}T_i$.
For $m=1$, the condition (ii) follows from (i). For that reason, 
$\clf_1(\clh)$ is the class of all commuting contractive operator 
pairs on $\clh$. 
For $(T_1,T_2)\in \clf_m(\clh)$
we show that their product contraction
$T=T_1T_2$ is an $m$-hypercontraction on $\clh$.
In other words, for any $m\in\N$, $\clf_m(\clh)$ contains 
contractive factors of $m$-hypercontractions on $\clh$.
In particular, this also provides a sufficient condition 
for the product of a pair of commuting contractions $(T_1,T_2)$ on $\clh$
to be an
$m$-hypercontraction and the sufficient 
condition is simply that $(T_1,T_2)\in \clf_m(\clh)$. This
sufficient condition is not necessary as 
we find counterexamples.
The goal of this article is to describe the class 
of contractive factors $\clf_m(\clh)$ of $m$-hypercontractions, 
completely. 
One such explicit descriptions 
we obtain is as follows. For a Hilbert space $\cle$, 
a bounded analytic function $\Phi:\D\to \clb(\cle)$
is a $\clb(\cle)$-valued Schur function on $\D$
if 
\[
 \text{sup}_{z\in\D}\|\Phi(z)\|\le 1.
\]

If $T$ is a $m$-hypercontraction on a Hilbert space $\clh$,
then the following        
are equivalent: 

(i) $T=T_1T_2$ for some $(T_1,T_2)\in\clf_m(\clh)$;

(ii) there exist a pair of commuting unitaries $(W_1,W_2)$ 
on a Hilbert space  $\clr$
with $W=W_1W_2$ and 
a pair of $\clb(\cle)$-valued Schur functions on $\D$
\[
 \Phi(z) = (P+zP^{\perp})U^*, \text{ and }
 \Psi(z)= U(P^{\perp} + zP),\quad  (z \in \D)
\]
corresponding to a triple 
$(\cle,U,P)$ consisting of a Hilbert space 
$\cle$, a unitary $U$ on $\cle$ and an orthogonal projection $P$
in $\clb(\cle)$ such that 
$\clq$ is a joint
$(M_z^*\oplus W^*, M_{\Phi}^* \oplus W_1^*, M_{\Psi}^* \oplus W_2^*)$-invariant subspace of $A^2_m(\cle) \oplus \clr $ and 
\[ 
T_1 \cong P_{\clq}(M_{\Phi} \oplus W_1)|_{\clq},
\,\, T_2 \cong P_{\clq}(M_{\Psi} \oplus W_2)|_{\clq},
\,\, T \cong P_{\clq}(M_z \oplus W)|_{\clq}.
\]  
Furthermore, if $T$ is a pure $m$-hypercontraction
then the Hilbert space $\clr=\{0\}$.


This in turn provides a similar factorization result for subnormal operators.
The above factorization result is obtained by
finding a suitable and explicit dilation of commuting 
contractive operator triples, of the form $(T_1,T_2, T_1T_2)$
for $(T_1,T_2)\in\clf_m(\clh)$,
 on some weighted Bergman space.
At the same time, the explicit dilation of triples relies on a
Douglus type dilation of $m$-hypercontractions
and a commutant lifting technique originally found in ~\cite{DSS}. 

The plan of the paper is as follows.
Section ~\ref{Douglas-Hyper} contains Douglus
type dilation for $m$-hypercontractions.
We study different properties of $\clf_m(\clh)$ in 
Section~\ref{fmh}.
In Section ~\ref{Product-Hyper}, we find a suitable explicit 
dilation for the class of factors in $\clf_m(\clh)$. 
This is then used to obtain several factorization 
results in Section~\ref{Factorization}.
In the last section, we find examples 
of factors of $m$-hypercontractions on $\clh$ which 
are not an element of $\clf_m(\clh)$.

\newsection{Douglas Type Dilation for Hypercontractions}
\label{Douglas-Hyper}

In this section, we find a Douglas type dilation and therefore
the model for $m$-hypercontractions as in Theorem~\ref{Agler},
 which is required to obtain dilation of factors of hypercontractions.
Our explicit construction of
 Douglas type dilation for $m$-hypercontractions seems to be new.
 We believe that this may be known to experts in the area.
But, we include the construction of such explicit dilation
for completeness.

Recall that a contraction $T$ on $\clh$ is a $m$-hypercontraction
if for all $n=1,\dots, m$,
\[K_n^{-1}(T,T^*)= \sum_{k=0}^n (-1)^k \left(
    \begin{array}{c}
      n \\
      k
    \end{array} \right)T^kT^{*k} \geq 0.\]
Also for all $n=1,\dots,m$, $n$-th order defect operator and defect space are 
\[
 D_{n,T}=K_n^{-1}(T,T^*)^{1/2}\ \text{and }
 \cld_{n,T}=\overline{\text{ran}}D_{n,T},
\]
respectively. The sequence of weights $\{w_{n,k}\}_{k= 0}^{\infty}$
 given by 
\[
 (1-x)^{-n} = \sum_{k=0}^{\infty} w_{n,k} x^k, \hspace{1cm}
  (|x| < 1, n\in\N\cup\{0\})
\]
 play a crucial role in what follows and we invoke a lemma from 
~\cite{JA2} which describe certain relationship of these weights for different 
values of $n$. 

\begin{Lemma}[cf \cite{JA2}]
\label{weight seq}
Let $\{w_{n,k}\}_{k\ge 0, n\ge 0}$ be as above.
Then for all $n,k \geq 1 $,
\[w_{n,k} - w_{n,k-1} = w_{n-1,k}.\]
\end{Lemma}

For a fixed $1\le n\le m$, consider the orthonormal basis
$\{\psi_{n,k}(z)= \sqrt{w_{n,k}}z^k\}_{k=0}^{\infty}$
for the weighted Bergman space $A^2_n$.
Then the kernel function of $A^2_n$ is given by 
\[ K_n(z,w)= (1-z\bar{w})^{-n} =
\sum_{k=0}^{\infty} \overline{\psi_{n,k}(w)} \psi_{n,k}(z)\quad (z,w\in\D). \] 
We set, for $r\geq0$,
\[ f_r^{(n)}(z,w): = \sum_{k=r}^{\infty} \psi_{n,k}(z)K_n^{-1}(z,w)\overline{\psi_{n,k}(w)}
\quad  (z,w\in\D). \]
Then it can be easily seen that $f_0^{(n)}\equiv 1$ and
\[   f_r^{(n)}(z,w) = 1 - \sum_{k=0}^{r-1} \psi_{n,k}(z)K_n^{-1}(z,w)\overline{\psi_{n,k}(w)},\  (r\ge 1)\]
and consequently, $f_r^{(n)}$ is a polynomial for all $r\ge 0$.
As a result, using  polynomial calculus,
we define \[f_r^{(n)}(T,T^*) := 1 - \sum_{k=0}^{r-1} w_{n,k} T^k K_n^{-1}(T,T^*)T^{*k},\quad (r\ge 0, 1\le n\le m)\]
for any $m$-hypercontraction $T$ on $\clh$.
These operator are used to study the canonical dilation map 
 $\Pi_{m,T}: \clh\to A^2_{m}(\cld_T)$ defined by 
\begin{align}\label{v_{m,T}}
(\Pi_{m,T}h)(z) = D_{m,T}(I_{\clh} - zT^*)^{-m}h,
 \hspace{1cm} (h\in \clh, z\in \D)
\end{align}
corresponding to an $m$-hypercontraction $T$ on $\clh$.
The next proposition shows that the operator $\Pi_{m,T}$
 is a contraction and it is 
analogous to Proposition 10 in ~\cite{AEM} 
for the case when $T$ is a pure $m$-hypercontraction.

\begin{Proposition}\label{norm identity}
In the above setting, we have the following:
\begin{enumerate}
\item [\textup{(i)}]
For any $1\le n\le m$, the sequence 
$\{f_r^{(n)}(T,T^*)\}_{r=0}^{\infty}$
is a decreasing sequence of positive operators. 
\item[\textup{(ii)}] 
$ \|\Pi_{m,T}h \|^2 = \|h\|^2 - \lim_{r \to \infty} \langle f_r^{(m)}(T,T^*)h,h \rangle\quad  (h\in \clh)$. 
\end{enumerate}
\end{Proposition}

\begin{proof}
It is clear from the definition that $\{f^{(n)}_r(T,T^*)\}_{r=0}^{\infty}$
is a decreasing sequence for all $n=1,\dots,m$.
For the positivity,
it follows from Lemma~\ref{weight seq} and the discussion 
succeeding it that
for $r\ge 0$ and $1\le n\le m$, 
\begin{align}
\label{fnr}
 f_r^{(n)}(T,T^*)
&=  1 - \sum_{k=0}^{r-1} w_{n,k} T^k K_n^{-1}(T,T^*)T^{*k}\nonumber\\
&= 1 - \sum_{k=0}^{r-1} w_{n,k} T^k \Big( K_{n-1}^{-1}(T,T^*)- TK_{n-1}^{-1}(T,T^*)T^* \Big) T^{*k}\nonumber\\
&= 1 - w_{n,0}K_{n-1}^{-1}(T,T^*) - \sum_{k=1}^{r-1} (w_{n,k} - w_{n,k-1})T^k K_{n-1}^{-1}(T,T^*)T^{*k}\nonumber\\
& \hspace{ 2in}+ w_{n,r-1}T^r K_{n-1}^{-1}(T,T^*)T^{*r} \nonumber\\
&= f_r^{(n-1)}(T,T^*) + w_{n,r-1}T^r K_{n-1}^{-1}(T,T^*)T^{*r}. 
\end{align}
Since $w_{n,r-1}T^r K_{n-1}^{-1}(T,T^*)T^{*r}\ge 0$, we conclude
that $f_r^{(n)}(T,T^*) \geq f_r^{(n-1)}(T,T^*)$ for all $r\ge 0$
and for all $n=1,\dots,m$.
As a result, we also have 
\[ f_r^{(n)}(T,T^*) \geq f_r^{(n-1)}(T,T^*) \geq \cdots \geq f_r^{(1)}(T,T^*)=T^rT^{*r} \geq 0.\] 
This proves that $\{f_r^{(n)}(T,T^*)\}_{r=0}^{\infty}$ is a
decreasing sequence of positive operators. 
The proof of $\textup{(ii)}$ is verbatim with the proof of Proposition 10 in \cite{AEM}.
\end{proof}

By the above result, 
we denote the strong operator limit
of the sequence $\{f^{(n)}_r(T,T^*)\}_{r=0}^{\infty}$ and its range as
\begin{equation} 
\label{Q}
Q_{n,T}^2 := \textup{SOT}-\lim_{r\to \infty} f_r^{(n)}(T,T^*) \quad
\clq_{n,T}= \overline{ran}Q_{n,T}\quad (1\le n\le m).
\end{equation}

It should be noted that if $T$ is a pure $m$-hypercontraction then 
\[
 \textup{SOT}-\lim_{r\to \infty}f_r^{(m)}(T,T^*) =
 \textup{SOT}-\lim_{r\to \infty}f_r^{(m-1)}(T,T^*) = 
 \cdots = 
 \textup{SOT}-\lim_{r\to \infty}f_r^{(1)}(T,T^*)= 0.
\]
This can derived from the identity ~\eqref{fnr} 
 and from the fact that
$w_{n,r-1}T^r K_{n-1}^{-1}(T,T^*)T^{*r}\to 0$ in the strong operator 
topology (see Lemma 2.11 in ~\cite{JA2}).
Thus the canonical dilation map $\Pi_{m,T}$ is an 
isometry if and only if $T$ is a pure $m$-hypercontraction. The intertwining 
property of $\Pi_{m,T}$, that is $\Pi_{m,T}T^*= M_z^*\Pi_{m,T}$
where $M_z$ is the 
shift on $A^2_m(\cld_{m,T})$, is evident from the definition of $\Pi_{m,T}$.

Before we present the main theorem of this section, we recall a
well-known factorization result due to Douglas.

\begin{Lemma}{\textup{(cf. \cite{Douglas})}} 
\label{Douglas-lemma} 
Let $A$ and $B$ be two bounded linear operators on a Hilbert space $\clh$. Then there exists a contraction $C$ on $\clh$ such that $A=BC$ if and only if \[ AA^* \leq BB^*.\]
\end{Lemma}
The explicit construction of Douglas type dilation for $m$-hypercontractions is given next.

\begin{Theorem}\label{Dilation}
If $T \in \clb(\clh)$ is an $m$-hypercontraction, then there exist
a Hilbert space $\clr$, an isometry $\Pi_T:\clh\to A^2_m(\cld_{m,T})\oplus\clr$ and a unitary $W$ on $\clr$
such that 
\[
\Pi_TT^*=(M_z^*\oplus W^*)\Pi_T.
\]
In particular, 
\[ T \cong P_{\clq}(M_z \oplus W)|_{\clq},\]
 where $\clq=\mbox{ran} \Pi_T$ is the $(M_z\oplus W)^*$-invariant subspace of
$A^2(\cld_{m,T}) \oplus \clr$.

\end{Theorem}

\begin{proof}
Let $Q_{n,T}$ be the positive operator as in ~\eqref{Q}
 for all $1\le n\le m$. 
By induction on $n$, we prove that 
\[  TQ_{n,T}^2T^* = Q_{n,T}^2\ (n=1,\dots,m) . \]
It is easy to see that it holds for $n=1$. 
Then we assume that the identity holds for some $n$ with $1\le n <m$. 
Thus by the assumption
$f_{r+1}^{(n)}(T,T^*) - Tf_{r+1}^{(n)}(T,T^*)T^* \to 0$
in the strong operator topology as $r\to\infty$.
Now,
\begin{align*}
& f_{r+1}^{(n+1)}(T,T^*) - Tf_r^{(n+1)}(T,T^*)T^* \\
&= I - TT^* - K_{n+1}^{-1}(T,T^*) + \sum_{k=0}^{r-1} (w_{n+1,k} - w_{n+1,k+1})T^{k+1} K_{n+1}^{-1}(T,T^*)T^{*(k+1)} \\
&=  I - TT^* - K_{n+1}^{-1}(T,T^*) - \sum_{k=0}^{r-1} w_{n,k+1} T^{k+1}K_{n+1}^{-1}(T,T^*)T^{*(k+1)} \\
&=  I - TT^* - (K_{n}^{-1}(T,T^*) - TK_n^{-1}(T,T^*)T^*)\\
& \hspace{2in}- \sum_{k=0}^{r-1} w_{n,k+1} T^{k+1}\big(K_{n}^{-1}(T,T^*) - 
TK_{n}^{-1}(T,T^*)T^*\big)T^{*(k+1)} \\
& = \big(I - \sum_{k=0}^r w_{n,k} T^kK_{n}^{-1}(T,T^*)T^{*k} \big) 
- \big( TT^* - \sum_{k=0}^r w_{n,k} T^{k+1}K_{n}^{-1}(T,T^*)T^{*(k+1)} \big) \\
& = f_{r+1}^{(n)}(T,T^*) - Tf_{r+1}^{(n)}(T,T^*)T^*.
\end{align*}
Consequently by the induction hypothesis, 
$f_{r+1}^{(n+1)}(T,T^*) - Tf_r^{(n+1)}(T,T^*)T^*\to 0$
in the strong operator topology as $r\to \infty$.
 This in turn implies that
\[
  TQ_{n+1,T}^2T^* = Q_{n+1,T}^2.
\]
Thus we have proved that $TQ_{n,T}^2T^* = Q_{n,T}^2$
for all $n=1,\dots,m$. 
In particular since $ TQ_{m,T}^2T^* = Q_{m,T}^2$, 
 by Lemma~\ref{Douglas-lemma}, 
there exists an isometry $X^*$ on $\clq_{m,T}$
such that 
\begin{align}
\label{X*}
 X^*Q_{m,T} = Q_{m,T}T^*.
\end{align}
Let $W^*$ on $\clr \supseteq \clq^{(m)}$
be the minimal unitary extension of $X^*$.
Then, by Proposition~\ref{norm identity},
the map $\Pi_T: \clh \to A^2_m(\cld_{m,T}) \oplus \clr$ defined by
\[ \Pi_Th= (\Pi_{m,T}h, Q_{m,T}h),\quad (h\in\clh)\]
is an isometry and it also satisfies 
\[
 \Pi_TT^* = (M_z \oplus W)^*\Pi_T.
\]
Here the intertwining property follows from ~\eqref{X*}.
Therefore, $\clq=\text{ran}\Pi_T$ is a
 $(M_z\oplus W)^*$-invariant subspace of
  $A^2_m(\cld_{m,T})\oplus\clr$
 and we have
\[ T \cong P_{\clq}(M_z \oplus W)|_{\clq} .\] 
This completes the proof.
\end{proof}

\section{The class $\clf_m(\clh)$}
\label{fmh}

The class of contractive factors $\clf_m(\clh)$ and its basic 
properties are studied in this section.
To begin with we recall the definition of the class  $\clf_m(\clh)$. 
A commuting pair of contractions $(T_1,T_2)$ on $\clh$ is an
 element of $\clf_m(\clh)$ if $K_{m-1}^{-1}(T,T^*)-T_iK_{m-1}^{-1}(T,T^*)T_i^*\ge 0$ for all $i=1,2$ where $T=T_1T_2$. 
 
 For $(T_1,T_2)\in\clf_m(\clh)$ with 
 $T=T_1T_2$, we fix 
 the following notations for the rest of the article:
 \begin{align}
 \label{dmtt}
  D_{n,T,T_i}^2=K_{n-1}^{-1}(T,T^*)-T_iK_{n-1}^{-1}(T,T^*)T_i^*
  \text{ and } \cld_{n,T,T_i}=\overline{\text{ran}}D_{n,T,T_i}^2
  \quad (n\in\N, i=1,2).
 \end{align}
 With the above notation, we have the following useful identity
 \begin{align}
 \label{dmtt identity}
  &D_{n,T,T_i}^2-TD_{n,T,T_i}^2 T^*\nonumber \\
  &=K_{n-1}^{-1}(T,T^*)-TK_{n-1}^{-1}(T,T^*)T^*-
  T_i\big(K_{n-1}^{-1}(T,T^*)-TK_{n-1}^{-1}(T,T^*)T^*\big)T_i^*\nonumber \\
  &=K_{n}^{-1}(T,T^*)-T_iK_{n}^{-1}(T,T^*)T_i^*\nonumber\\
  &=D_{n+1,T,T_i}^2,
 \end{align}
for all $n\ge 0$.
Next we show an intermediate positivity type result.
\begin{Lemma}
If $(T_1,T_2)\in\clf_m(\clh)$ then $(T_1,T_2)\in \clf_n(\clh)$
for all $1\le n\le m$. 
\end{Lemma} 
\begin{proof}
 It is enough to show that $D_{n,T,T_i}^2\ge 0$ for all 
 $n=1,\dots,m$ and for all $i=1,2$. We only consider the case 
  $i=1$ as it is symmetrical for $i=2$. By the hypothesis 
 $D_{m,T,T_1}^2\ge 0$ and $D_{1,T,T_1}^2\ge 0$.
 To show $D_{(m-1),T,T_1}^2\ge 0$, we assume $m\ge 2$ and
 consider the sequence $\{a_r\}_{r=0}^{\infty}$,
 corresponding to a fixed $h\in\clh$, defined as 
 \[
  a_r=\langle T^rD_{(m-1),T,T_1}^2T^{* r}h, h\rangle\quad (r\ge 0).
 \]
Then for any $r\ge 0$, using ~\eqref{dmtt identity}, we have
\begin{align*}
 a_r-a_{r+1}
 &=
 \langle T^r(D_{(m-1),T,T_1}^2-TD_{(m-1),T,T_1}^2T^*)T^{* r}h,h\rangle\\
 &=\langle T^r D_{m,T,T_1}^2T^{* r}h, h\rangle \ge 0.
\end{align*}
Thus $\{a_r\}_{r=0}^{\infty}$ is a decreasing sequence. Also since 
\begin{align*}
 \Big|\sum_{r=0}^Na_r\Big|
 &=\Big|\Big\langle\sum_{r=0}^N T^r(D_{(m-2),T,T_1}^2-TD_{(m-2),T,T_1}^2T^*)
 T^{* r}h,h\Big\rangle\Big|\\
 &=\Big|\Big\langle (D_{(m-2),T,T_1}^2- T^{N+1}D_{(m-2),T,T_1}^2T^{* (N+1)})h,
 h\Big\rangle\Big|\\
 & \le 2\|h\|^2\|D_{(m-2),T,T_1}^2\|,
\end{align*}
$a_r\ge 0$ for all $r\ge 0$. In particular, it implies that 
$ D_{(m-1),T,T_1}^2\ge 0$.  Therefore, by induction on $m$,
we have all the required positivity. This completes the proof.
\end{proof}


Needless to say that the product of two commuting contractions is not 
an $m$-hypercontraction, in general. 
We find a sufficient condition for product of two commuting 
contractions to be an $m$-hypercontraction. 
The sufficient condition is simply that the pair of contractions on $\clh$ 
should be an element of $\clf_m(\clh)$.  
This is proved in the next lemma, which is in the same spirit 
as Lemma 3.1 in  ~\cite{BDHS}.
\begin{Lemma}
\label{sufficient}
 If $(T_1,T_2)\in \clf_m(\clh)$, then 
 $T_1T_2$ is a $m$-hypercontraction.
\end{Lemma}
\begin{proof}
 Let $T=T_1T_2$. The proof is obvious for $m=1$. For $m\ge 2$
 note that 
 \begin{align*}
 &K_m^{-1}(T,T^*)\\
 &= K_{m-1}(T,T^*)-TK_{m-1}^{-1}(T,T^*)T^*\\
 &=\big(K_{m-1}^{-1}(T,T^*)-T_1^*K_{m-1}^{-1}(T,T^*)T_1^*\big)
 + T_1\big(K_{m-1}^{-1}(T,T^*)-T_2^*K_{m-1}^{-1}(T,T^*)T_2^*\big)T_1^*
 \ge 0.
 \end{align*}
 This completes the proof. 
\end{proof}
The converse of this lemma is not true as we find 
 counterexamples in last section of this article. This suggests that 
 $\clf_m(\clh)$ does not contain all the factors of 
 $m$-hypercontractions. 
Before going further, we consider elementary examples 
of elements in $\clf_m(\clh)$. These examples 
are based on a triple $(\cle, U,P)$ consists of a Hilbert space $\cle$, a unitary operator $U$ on $\cle$ and an orthogonal projection
 $P$ in $\clb(\cle)$. For such a triple, 
the $\clb(\cle)$-valued analytic functions
\[
 \Phi(z)= (P+zP^{\perp})U^*, \text{ and } \Psi(z)=U(P^{\perp}+zP)\quad (z\in\D)
\]
are easily seen to be Schur functions on $\D$,
 that is they are in the unit ball 
of the Banach algebra $H^{\infty}_{\clb(\cle)}(\D)$ consists of bounded
$\clb(\cle)$-valued analytic functions on $\D$.
It is easy to see that 
\[
\Phi(z)\Psi(z)=\Psi(z)\Phi(z)=zI_{\cle}\quad (z\in\D).
\]
We refer to $\Phi, \Psi$ as {\em canonical pair of Schur functions}
 on $\D$ corresponding 
to the triple $(\cle, U,P)$.  
 We claim that the commuting
pair of multiplication operators $(M_{\Phi}, M_{\Psi})$ on $A^2_m(\cle)$
is an element of $\clf_m(A^2_m(\cle))$. Indeed, if $\cle_1=\text{ran}P$
and $\cle_2=\text{ran}P^{\perp}$ then $\cle=\cle_1\oplus \cle_2$.
With respect to the above decomposition of the co-efficient space, we have 
$A^2_m(\cle)= A^2_m(\cle_1)\oplus A^2_m(\cle_2)$
and 
\begin{align*}
 &K_{m-1}(M_z,M_z^*)-M_{\Phi}K_{m-1}(M_z,M_z^*)M_{\Phi}^*\\
 &=K_{m-1}(M_z,M_z^*)-
 \begin{bmatrix}
I_{A^2_m(\cle_1)} & 0\cr 0 & M_z\otimes I_{\cle_2}
\end{bmatrix}(I\otimes U^*)K_{m-1}(M_z,M_z^*)(I\otimes U)\begin{bmatrix}
I_{A^2_m(\cle_1)} & 0\cr 0 & M_z^*\otimes I_{\cle_2}
\end{bmatrix}\\
&=K_{m-1}(M_z,M_z^*)-
 \begin{bmatrix}
I_{A^2_m(\cle_1)} & 0\cr 0 & M_z\otimes I_{\cle_2}
\end{bmatrix}K_{m-1}(M_z,M_z^*)\begin{bmatrix}
I_{A^2_m(\cle_1)} & 0\cr 0 & M_z^*\otimes I_{\cle_2}
\end{bmatrix}\\
& =\begin{bmatrix}
  0& 0 \cr 0 & K_m(M_z\otimes I_{\cle_2},M_z^*\otimes I_{\cle_2})
 \end{bmatrix}\ge 0,
\end{align*}
as $M_z\otimes I_{\cle_2}$ on $A^2_m(\cle_2)$
 is an $m$-hypercontraction. 
Similarly, we have
\[ K_{m-1}(M_z,M_z^*)-M_{\Psi}K_{m-1}(M_z,M_z^*)M_{\Psi}^*\ge 0.\]
This proves the claim. In fact we will see below that 
any pair $(T_1, T_2)\in\clf_m(\clh)$ with $T_1T_2$ is pure 
dilates to such a pair $(M_{\Phi}, M_{\Psi})$ on some $A^2_m(\cle)$,
and therefore they serve as a model for a class of factors
of pure $m$-hypercontractions. 


\section{Dilation of factors}\label{Product-Hyper}

Our main concern is to propose a model for the class  $\clf_m(\clh)$ of factors of $m$-hypercontractions. This is achieved by finding an explicit 
dilation of a triple of commuting contractions $(T_1,T_2,T_1T_2)$ on some
weighted Bergman space, where $(T_1,T_2)\in\clf_m(\clh)$. 
We say an $n$-tuple of commuting contractions
$(T_1,\dots,T_n)$ on $\clh$ dilates to a commuting $n$-tuple 
of operators $(R_1,\dots,R_n)$ on $\clk$ if there is an isometry 
$\Pi:\clh\to\clk$, such that 
\[
 \Pi S_i^*=R_i^*\Pi\quad (i=1,\dots,n).
\]
The map $\Pi$ is often refer as the dilation map.

We prove a lemma which will be the key to the dilation results
 obtained 
in this section. This is analogous to Theorem 2.1 in ~\cite{DSS}. Let $(T_1,T_2)\in\clf_m(\clh)$.  Since $T=T_1T_2$
 is an $m$-hypercontraction, recall the canonical dilation map
 $\Pi_{m,T}:\clh \to A^2_m(\cld_{m,T})$ defined by 
 \[
( \Pi_{m,T}h)(z)=D_{m,T}(I-zT^*)^{-m}h\quad (h\in\clh, z\in\D),
 \]
 such that $\Pi_{m,T}T^*=M_{z}^*\Pi_{m,T}$.  If $V:\cld_{m,T}\to \cle$
 is a isometry for some Hilbert space $\cle$, then the map 
 \[
 \Pi_{V}:=(I\otimes V)\Pi_{m,T}: \clh \to A^2_m(\cle)
 \]
also intertwines with $T^*$ and $M_z^*$ on $A^2_m(\cle)$,
that is $\Pi_V T^*=M_z^*\Pi_V$.
\begin{Lemma}\label{transfer fn}
With the above notation, if $\cld$ is a Hilbert space and if 
\[ U_i= \begin{bmatrix}
        A_i  &   B_i  \cr
        C_i  &   0  \cr
        \end{bmatrix}: \cle \oplus (\cld\oplus \cld_{m,T, T_i})
        \to  \cle \oplus  (\cld\oplus\cld_{m,T,T_i})\quad (i=1,2)  \] 
   is a unitary operator such that for all $h\in\clh$,
\begin{align*} 
& U_i\big(VD_{m,T}h, 0_{\cld}, D_{m,T,T_i}T^*h \big)
= \big(VD_{m,T}T_i^*h, 0_{\cld}, D_{m,T,T_i}h \big), \  (i=1,2)
\end{align*}
then the  $\clb(\cle)$-valued Schur function
$\Phi_i(z)= A_i^* + zC_i^*B_i^*\ (z\in\D) $,
transfer function corresponding to the unitary  $U_i^*$, satisfies  
\[ \Pi_VT_i^* = M_{\Phi_i}^*\Pi_V,\] 
for all $i=1,2$.
\end{Lemma} 

\begin{proof}
 Since
\begin{align*}
 \begin{bmatrix}
A_i  &  B_i \cr
C_i  &   0  \cr
\end{bmatrix} \begin{bmatrix}
VD_{m,T}h \cr (0_{\cld}, D_{m,T,T_i}T^*h)
\end{bmatrix} 
= \begin{bmatrix}
VD_{m,T}T_i^*h \cr (0_{\cld}, D_{m,T, T_i}h)
\end{bmatrix}, \quad (h\in\clh, i=1,2)
\end{align*}
we have
\[ A_iVD_{m,T}h + B_i (0_{\cld},D_{m,T,T_i}T^*h) 
= VD_{m,T}T_i^*h , \quad C_iVD_{m,T}h =(0_{\cld}, D_{m,T,T_i}h),
\]
for all $h\in\clh$ and $i=1,2$. Simplifying further, we get
\[ VD_{m,T}T_i^* = A_iVD_{m,T} + B_iC_iVD_{m,T}T^* \]
for all $i=1,2$.
Finally, if $n\geq 1$, $h\in \clh$ and $\eta \in \cle$ then 
\begin{align*}
\langle M_{{\Phi}_i}^* \Pi_Vh, z^n \eta \rangle &= \langle (I \otimes V)D_{m,T}(1-zT^*)^{-m}h, (A_i^*+ zC_i^*B_i^*)(z^n \eta) \rangle \\
&= \langle (A_iVD_{m,T} + B_iC_iVD_{m,T^*}T^*)T^{*n}h, \eta \rangle \\ 
&= \langle VD_{m,T}T_i^*(T^{*n}h), \eta \rangle \\
&= \langle \Pi_VT_i^*h, z^n \eta \rangle, \quad (i=1,2).
\end{align*}
Therefore, we get $\Pi_VT_i^* = M_{\Phi_i}^*\Pi_V$ for all $i=1,2$.
This ends the proof.

\end{proof}

 Let $(T_1,T_2)\in\clf_m(\clh)$ with $T=T_1T_2$.
The following identity, as in the proof of Lemma~\ref{sufficient},   
\[
K_m^{-1}(T,T^*)=D_{m,T,T_1}^2 + T_1D_{m,T,T_2}^2T_1^* =
 D_{m,T,T_2}^2 + T_2D_{m,T,T_1}^2T_2^*,\]
implies that for all $h\in\clh$,
\[
 \|D_{m,T,T_1}h\|^2+\|D_{m,T,T_2}T_1^*h\|^2=\|D_{m,T,T_2}h\|^2+
 \|D_{m,T,T_1}T_2^* h\|^2.
\]

This leads us to define isometries 
\[ U: \{D_{m,T,T_1}h \oplus D_{m,T,T_2}T_1^*h: h \in \clh \} \rightarrow
\{D_{m,T,T_2}h \oplus D_{m,T,T_1}T_2^*h: h \in \clh \} \] 
defined by 
\begin{equation}\label{special U} 
U \big( D_{m,T,T_2}h , D_{m,T,T_1}T_2^*h \big)
= \big( D_{m,T,T_1}h, D_{m,T,T_2}T_1^*h \big),\quad (h\in\clh)
\end{equation}  
and $V:\cld_{m,T}\to \cld_{m,T,T_1}\oplus \cld_{m,T,T_2}$ defined by 
\begin{equation}
 \label{special V}
 V(D_{m,T}h)=(D_{m,T,T_1}h, D_{m,T,T_2}T_1^* h)\quad (h\in\clh). 
\end{equation}
We are now ready to prove the explicit dilation result for the pure case.

\begin{Theorem}\label{pure-dilation}
Let $(T_1,T_2) \in \clf_m(\clh) $ be such that
 $T=T_1T_2$ is a pure contraction.
Then there exist a triple $(\cle, U,P)$ consists of a Hilbert space 
$\cle$, a unitary $U$ on $\cle$ and a projection $P$ in $\mathcal B (\cle)$
and an isometry $\Pi:\clh\to A^2_m(\cle)$ such that 
\[
\Pi T_1^*=M_{\Phi}^*\Pi,\  \Pi T_2^*=M_{\Psi}^*\Pi, \text { and }
\Pi T^*=M_z^*\Pi
\]
where 
$\Phi$ and $\Psi$ are the $\clb(\cle)$-valued canonical Schur functions on $\D$
corresponding to the triple $(\cle,U,P)$ given by 
\[\Phi(z)=(P+zP^{\perp})U^*\ \text{ and }\ \Psi(z)= U(P^{\perp}+zP)
\]
for all $z\in\D$.

In particular, $\clq=\mbox{ran}\Pi$ is a joint
 $(M_{\Phi}^*,M_{\Psi}^*, M_{z}^*)$-invariant subspace of
  $A^2_m(\cle)$ such that
  
 \[
 T_1\cong P_{\clq}M_{\Phi}|_{\clq},\  T_2\cong P_{\clq}M_{\Psi}|_{\clq}
 \text{ and } T\cong P_{\clq}M_z|_{\clq}.
 \]
\end{Theorem}

\begin{proof}
We first consider the isometry $U$ as in ~\eqref{special U}
and by adding an infinite dimensional Hilbert space $\cld$,
if necessary, we extend it to a unitary on
$\cle:=(\cld\oplus \cld_{m,T,T_1})\oplus \cld_{m,T,T_2}$.
We continue to denote the 
unitary by $U$, and therefore we have a unitary $U:\cle\to \cle $ which satisfies 
\[
 U\big(0_{\cld},D_{m,T,T_1}T_2^*h, D_{m,T,T_2}h\big)
= \big(0_{\cld}, D_{m,T,T_1}h, D_{m,T,T_2}T_1^*h \big),\quad (h\in\clh).
\]
Also we view the isometry $V$ in ~\eqref{special V}, as an isometry
$V: \cld_{m,T}\to \cle$ defined by 
\[
 V(D_{m,T}h)=(0_{\cld},D_{m,T,T_1}h, D_{m,T,T_2}T_1^* h)\quad (h\in\clh).
\]
Since $T$ is a pure $m$-hypercontraction, then the canonical dilation map 
$\Pi_{m,T}: \clh\to A^2_m(\cld_{m,T})$ is an isometry, and as a result
\begin{equation}
\label{PiV}
 \Pi_V=(I\otimes V)\Pi_{m,T}: \clh\to A^2_m(\cle)
\end{equation}
is also an isometry. The isometry $\Pi_V$ will be the dilation map in this 
context. 

To complete the proof of the theorem, we construct 
unitaries which satisfy the hypothesis of Lemma~\ref{transfer fn}.
To this end, we consider the inclusion maps 
$\iota_1: \cld\oplus\cld_{m,T,T_1} \to \cle $ and 
$\iota_2:\cld_{m,T,T_2}\to\cle$ defined by 
\[ 
\iota_1(h,k_1) = (h,k_1,0) \ \text{ and } i_2(k_2)= (0,0,k_2),
(h\in \cld, k_1\in\cld_{m,T,T_1}, k_2\in\cld_{m,T,T_2}).
\]
We also consider the orthogonal projection $P=\iota_2\iota_2^*$.
Then it is easy to check that 
\[
 \begin{bmatrix}
   P & \iota_1 \cr \iota_1^* & 0
\end{bmatrix}: \cle \oplus (\cld\oplus\cld_{m,T,T_1}) \to
\cle \oplus (\cld\oplus\cld_{m,T,T_1}) 
\] 
and 
\[
\begin{bmatrix}
 P^{\perp} & \iota_2 \cr \iota_2^* & 0
\end{bmatrix}: \cle \oplus \cld_{m,T,T_2} \to \cle \oplus \cld_{m,T,T_2}  
\]
are unitary.
The unitary 
\[ U_1:= \begin{bmatrix}
U &  0 \cr 0 & I 
\end{bmatrix} \begin{bmatrix}
P  &  i_1 \cr  i_1^* &  0 
\end{bmatrix}= \begin{bmatrix}
UP  &  Ui_1  \cr  i_1^*  &  0 
\end{bmatrix}: \cle \oplus (\cld\oplus\cld_{m,T,T_1}) \to 
\cle \oplus (\cld\oplus\cld_{m,T,T_1}), \] 
satisfies
\begin{align*}
U_1
\begin{bmatrix}
VD_{m,T}h \cr D_{m,T,T_1}T^*h
\end{bmatrix}
&= 
\begin{bmatrix}
UP  &  Ui_1  \cr  i_1^*  &  0 
\end{bmatrix} 
\begin{bmatrix}
VD_{m,T}h \cr D_{m,T,T_1}T^*h
\end{bmatrix} \\
&= \begin{bmatrix}
U(0_{\cld}, D_{m,T,T_1}T_2^*T_1^*h, D_{m,T,T_2}T_1^*h ) 
\cr (0_{\cld},D_{m,T,T_1}h)
\end{bmatrix} \\
&= \begin{bmatrix}
(0_{\cld}, D_{m,T,T_1}T_1^*h, D_{m,T,T_2}T_1^{*2}h )
\cr (0_{\cld},D_{m,T,T_1}h)
\end{bmatrix} \\
&= \begin{bmatrix}
VD_{m,T}T_1^*h \cr (0_{\cld}, D_{m,T,T_1}h)
\end{bmatrix},
\end{align*}
for all $h\in\clh$.
Subsequently, a similar computation also shows that the unitary 
\[ U_2:= \begin{bmatrix}
P^{\perp}  &  \iota_2 \cr  \iota_2^* &  0 
\end{bmatrix} \begin{bmatrix}
U^* &  0 \cr 0 & I 
\end{bmatrix} : \cle \oplus \cld_{m,T,T_2} \to
\cle \oplus \cld_{m,T,T_2},\]
satisfies
\[
 U_2(V\cld_{m,T}h, D_{m,T,T_2}T^*h)= 
 (VD_{m,T}T_2^* h, D_{m,T,T_2}h)
\]
for all $h\in \clh$.
The proof now follows by appealing Lemma~\ref{transfer fn}
for the unitaries $U_1$ and $U_2$. 
\end{proof}
 
 \begin{Remark}
 \label{R-pure}
The converse of the above theorem is also true. That is, if $(T_1,T_2,T)$
is a triple of commuting contractions on $\clh$ and if $(T_1,T_2,T)$ dilates 
to $(M_{\Phi}, M_{\Psi}, M_z)$ on $A^2_m(\cle)$ for some Hilbert space $\cle$
where $\Phi$ and $\Psi$ are $\clb(\cle)$-valued canonical 
Schur functions on $\D$ corresponding to a  
 triple $(\cle, U,P)$, then $(T_1,T_2)\in\clf_m(\clh)$
and $T=T_1T_2$. This follows immediately from the fact that
$(M_{\Phi}, M_{\Psi}, M_z)\in \clf_m(A^2_m(\cle))$.
 \end{Remark}


Having obtained the explicit dilation for the pure case, we now 
drop the pure assumption and find dilation for the general case.
\begin{Theorem}\label{general-dilation}
Let $(T_1,T_2) \in \clf_m(\clh)$ with $T=T_1T_2$. Then there exist 
a triple $(\cle, U,P)$ consists of a Hilbert space $\cle$, a unitary $U$
 on $\cle$ and an orthogonal projection $P$ in $\clb(\clh)$, a Hilbert space 
 $\clr$, a pair of commuting unitaries $(W_1,W_2)$ 
 on a Hilbert space $\clr$ with
  $W=W_1W_2$ and an isometry $\Pi:\clh\to A^2_m(\cle)$ such that
\[
\Pi T_1^*=(M_{\Phi}\oplus W_1)^*\Pi,\ 
\Pi T_2^*=(M_{\Psi}\oplus W_2)^*\Pi\
\text {and } \Pi T^*= (M_z\oplus W)^*\Pi
\]  
 where $\Phi$ and $\Psi$ are the $\clb(\cle)$-valued 
 canonical Schur function on $\D$ corresponding to the 
 triple $(\cle, U,P)$ given by 
  \[
  \Phi(z)=(P+zP^{\perp})U^* \ \text{ and } \Psi(z)=U(P^{\perp}+zP)
  \] 
  for all $z\in\D$. 
  
 In particular,  $\clq=\mbox{ran}\Pi$ is a joint
  $(M_z^*\oplus W^*, M_{\Phi}^*\oplus W_1^*, M_{\Psi}^*\oplus W_2^*)$-invariant subspace of $A^2_m(\cle)\oplus \clr$ such that
\[
T_1\cong P_{\clq}(M_{\Phi}\oplus W_1)|_{\clq}, T_2\cong P_{\clq}(M_{\Psi}\oplus W_2)|_{\clq}
\ \text{ and } T\cong P_{\clq}(M_z\oplus W)|_{\clq}.
\]
\end{Theorem}

\begin{proof}
Let $(\cle, U, P)$ be as in Theorem~\ref{pure-dilation}, and let $V$ be
as in ~\eqref{special V}. Then by the same way as it is done in the  
proof of Theorem~\ref{pure-dilation}, we have 
\begin{align}
\label{first intertwine relation}
\Pi_VT_1^*= M_{\Phi}^*\Pi_V, \Pi_VT_2^*=M_{\Psi}^*\Pi_V\ \text{ and }
\Pi_VT^*=M_z^*\Pi_V,
\end{align}  
where $\Phi(z)=(P+zP^{\perp})U^*$ and $\Psi(z)=U(P^{\perp}+zP)$
for all $z\in\D$, $\Pi_V=(I\otimes V)\Pi_{m, T}$ and
 $\Pi_{m,T}:\clh\to A^2_m(\cld_{m,T}), h\mapsto D_{m,T}(I-zT^*)^{-m}h$
  is the canonical dilation map. However, note that $\Pi_V$ is not an isometry 
  in general. To make it an isometry we follow the construction 
  done in Theorem~\ref{Dilation}.

Let $Q_{m,T}$ be the positive operator defined in ~\eqref{Q} by taking 
strong operator limit of the decreasing sequence of positive operators
$\{f^{(m)}_r(T,T^*)\}_{r=0}^{\infty}$ where 
\[
  f_r^{(m)}(T,T^*) = 1 - \sum_{k=0}^{r-1} w_{m,k} T^k K_m^{-1}(T,T^*)T^{*k}
  \quad (r\ge 0).
\]
It also follows from the proof of Theorem~\ref{Dilation} that 
\[
 TQ_{m,T}^2T^*=Q_{m,T}^2.
\]
We claim here that $Q_{m,T}^2\ge T_iQ_{m,T}T_i^*$ for all $i=1,2$. We prove 
the inequality for $i=1$ as the proof is similar for $i=2$. To this end,
it is enough to show that 
\[f_r^{(m)}(T,T^*)- T_1f_r^{(m)}(T,T^*)T_1^*\ge 0\]
for all $r\ge 0$. For a fixed $r\ge 0$, we use induction 
on $m$ to establish it. Since $f_r^{(1)}(T,T^*)= T^rT^{* r}$, 
it is easy to see that the inequality holds for $m=1$. We 
assume that for some $1\le n<m$,
$f_r^{(n)}(T,T^*)- T_1f_r^{(n)}(T,T^*)T_1^*\ge 0$.
Then 
\begin{align*}
 & f_r^{(n+1)}(T,T^*) - T_1f_r^{(n+1)}(T,T^*)T_1^*\\ 
 &= 1- T_1T_1^* + \sum_{k=0}^{r-1} w_{n+1,k}T^k ( T_1K_{n+1}^{-1}(T,T^*)T_1^* - K_{n+1}^{-1}(T,T^*)) T^{*k}\\
 &= Y_{n+1}-T_1Y_{n+1}T_1^*,
\end{align*}
where 
\begin{align*}
 &Y_{n+1}\\
 &= 1 - \sum_{k=0}^{r-1} w_{n+1,k} T^k(K_{n}^{-1}(T,T^*) - TK_{n}^{-1}(T,T^*))T^{*k} \\
&= 1 - K_{n}^{-1}(T,T^*) - \sum_{k=1}^{r-1} (w_{n+1,k} -
w_{n+1,k-1}) T^kK_{n}^{-1}(T,T^*)T^{*k} + w_{n+1,r-1}T^rB_{n}^{-1}(T,T^*)T^{*r}\\
&= \big( 1 - \sum_{k=0}^{r-1} w_{n,k}T^k K_{n}^{-1}(T,T^*)T^{*k} \big) +
w_{n+1,r-1} T^r K_n^{-1}(T,T^*)T^{*r}\\
&= f_r^{(n)}(T,T^*) + w_{n+1,r-1} T^r K_n^{-1}(T,T^*)T^{*r}.
\end{align*}

Hence we have 
\begin{align*}
& f_r^{(n+1)}(T,T^*) - T_1f_r^{(n+1)}(T,T^*)T_1^* \\
&= \big(f_r^{(n)}(T,T^*) - T_1f_r^{(n)}(T,T^*)T_1^*\big) +
w_{n+1,r-1} T^r \big( K_n^{-1}(T,T^*) - T_1K_n^{-1}(T,T^*)T_1^* \big) T^{*n}\ge 0.\end{align*}
Here we have used the fact that
$ K_n^{-1}(T,T^*) - T_1K_n^{-1}(T,T^*)T_1^*\ge 0$ for all $n=1,\dots,m$.
 This establishes our claim and therefore,
  $Q_{m,T}^2\ge T_iQ^2_{m,T}T_i^*$ for all $i=1,2$.
Then by Lemma~\ref{Douglas-lemma}, there exists a contraction 
$X_i$ on $\clq_{m,T}$ such that 
\begin{align}
\label{second intertwine relation}
 X_i^*Q_{m,T}=Q_{m,T}T_i^* \quad (i=1,2).
\end{align}
Further, since $Q_{m,T}^2=TQ_{m,T}^2T^*$, there is an isometry $X^*$
on $\clq_{m,T}$ such that $X^*Q_{m,T}=Q_{m,T}T^*$.
 It is now evident that 
$X^*=X_1^*X_2^*=X_2^*X_1^*$, and therefore $X_i^*$ is also an isometry
for all $i=1,2$. Let $(W_1^*,W_2^*, W^*)$ on $\clr\supset \clq_{m,T}$
be the minimal unitary extension of $(X_1^*,X_2^*,X^*)$
with
$W^*=W_1^*W_2^*$.

Following Theorem~\ref{Dilation}, consider the map 
 $\Pi: \clh \to A^2_m(\cle) \oplus \clr $ defined by 
\[ \Pi(h)= (\Pi_Vh , Q_{m,T}h),\quad (h\in\clh).\]
Then, by Proposition~\ref{norm identity} and the 
fact that $V$ is an isometry, it follows that $\Pi$ is an isometry. 
 Moreover, it follows from the relations ~\eqref{first intertwine relation}
  and ~\eqref{second intertwine relation} that 
\[
 \Pi T_1^* = (M_{\Phi}^* \oplus W_1^*)\Pi,
\Pi T_2^* = (M_{\Psi}^* \oplus W_2^*)\Pi\ 
\text{ and } \Pi T^*=(M_{z}^*\oplus W^*)\Pi.
\]
This completes the proof of the theorem.

\end{proof}

We conclude the section with a remark which 
is similar to the pure case.
\begin{Remark}
\label{R-general}
 The converse of the above theorem is also true. Naturally,
this follows from the fact that
 $(M_{\Phi}\oplus W_1, M_{\Psi}\oplus W_2)\in
  \clf_m(A^2_m(\cle)\oplus\clr)$. 
 \end{Remark}

\section{Factorization of hypercontractions}
\label{Factorization}
Combining the dilation results, Theorem~\ref{pure-dilation} and Theorem~\ref{general-dilation}, obtained in the previous section
with Remark~\ref{R-pure} and Remark~\ref{R-general}, 
we get the following immediate characterization of factors 
in the class $\clf_m(\clh)$.

\begin{Theorem}
Let $(T_1,T_2)$ be a pair of contractions on $\clh$. Then the 
following are equivalent:

$\textup{(i)}$ $(T_1,T_2)\in\clf_m(\clh)$;

$\textup{(ii)}$  there exist a pair of commuting unitaries  $(W_1, W_2)$
on a Hilbert space $\clr$ with $W=W_1W_2$ and  $\clb(\cle)$-valued 
 canonical Schur functions 
\[
\Phi(z) = (P+zP^{\perp})U^*\ \text{ and }
\Psi(z)= U(P^{\perp} + zP)\quad (z \in \D) 
\]
corresponding to a 
triple $(\cle,U,P)$ consisting of a Hilbert space 
$\cle$, a unitary $U$ and an orthogonal projection $P$ in $\clb(\cle)$
such that $\clq$ is a joint
$(M_z^*\oplus W^*, M_{\Phi}^* \oplus W_1^*, M_{\Psi}^* \oplus W_2^*)$-invariant subspace of $A^2_m(\cle) \oplus \clr$,
\[ 
T_1 \cong P_{\clq}(M_{\Phi} \oplus W_1)|_{\clq},
T_2 \cong P_{\clq}(M_{\Psi} \oplus W_2)|_{\clq}, 
\text{ and }T \cong P_{\clq}(M_z \oplus W)|_{\clq}.
\] 
 
 In particular, if $T_1T_2$ is a pure contraction then 
 the Hilbert space $\clr=\{0\}$.
\end{Theorem}

It is now clear that the above theorem is obtained by realizing a factor 
$(T_1,T_2)\in\clf_m(\clh)$ on the dilation space $A^2_m(\cle)\oplus \clr$
of $T=T_1T_2$. However, one would expect to realize $(T_1,T_2)$ on the
canonical dilation space of $T$ as in Theorem~\ref{Dilation}. 

To this end, we first consider $(T_1,T_2)\in\clf_m(\clh)$
with $T=T_1T_2$ is a pure contraction. Let $\Pi_V$ be the dilation 
map as in Theorem~\ref{pure-dilation}, that is 
\[
\Pi_VT_1^*=M_{\Phi}^*\Pi_V, \Pi_VT_2^*=M_{\Psi}^*\Pi_V\ \text{ and }
\Pi_VT^*=M_z^*\Pi_V
\]   
and, by ~\eqref{PiV}, 
\[
\Pi_V=(I\otimes V)\Pi_{m,T},
\]
where $\Pi_{m,T}$ is the isometric canonical dilation map corresponding 
to the pure $m$-hypercontraction $T$ and $V:\cld_{m,T}\to \cle$ is an isometry. Then, by the from of $\Pi_V$, the 
above intertwining relations yield 
\[
 \Pi_{m,T}T_1^*=(I\otimes V^*)M_{\Phi}^*(I\otimes V)\Pi_{m,T}
 \ \text { and }
  \Pi_{m,T}T_2^*=(I\otimes V^*)M_{\Psi}^*(I\otimes V)\Pi_{m,T}.
\]
Set 
\[\tilde{\Phi}(z):=V^*\Phi(z)V\ \text{ and } \tilde{\Psi}(z):=V^*\Psi(z)V,
\quad (z\in\D).\]
Then $\tilde{\Phi}$ and $\tilde{\Psi}$ are $\clb(\cld_{m,T})$-valued 
Schur functions on $\D$ such that 
\[
 \Pi_{m,T}T_1^*=M_{\tilde{\Phi}}^*\Pi_{m,T},
 \Pi_{m,T}T_2^*=M_{\tilde{\Psi}}^*\Pi_{m,T}.
\]
Observant reader may have noticed that 
$\tilde{\Phi}$ and $\tilde{\Psi}$ do not commute, in general.
However, $P_{\clq}M_{\tilde{\Phi}}|_{\clq}$ and
 $P_{\clq}M_{\tilde{\Psi}}|_{\clq}$ commutes and 
 \[
 P_{\clq}M_z|_{\clq}=P_{\clq}M_{\tilde{\Phi}}M_{\tilde{\Psi}}|_{\clq}
 =P_{\clq}M_{\tilde{\Psi}}M_{\tilde{\Phi}}|_{\clq},
 \]
 where $\clq=\mbox{ran}\Pi_{m,T}$.
Thus we have proved the following:
\begin{Theorem}
 Let $T$ be a pure $m$-hypercontraction on $\clh$. Then the following 
 are equivalent.
 \begin{enumerate}
  \item[\textup{$(i)$}]
$T=T_1T_2$ for some $(T_1,T_2)\in\clf_m(\clh)$;
 \item [\textup{$(ii)$}]
 there exist $\clb(\cld_{m,T})$-valued Schur functions
 \[
\tilde{\Phi}(z)=V^*(P+zP^{\perp})U^*V,\ \text{ and } 
\tilde{\Psi}(z)= V^*U(P^{\perp}+zP)V
\quad (z\in\D)
 \]
for some Hilbert space $\cle$, isometry $V:\cld_{m,T}\to \cle$,
unitary $U:\cle\to \cle$ and 
projection $P$ in $\clb(\cle)$ such that $\clq$ is a joint
 $(M_{\tilde{\Phi}}^*, M_{\tilde{\Psi}}^*)$-invariant subspace of $A^2_m(\cld_{m,T})$,
 \[
  P_{\clq}M_z|_{\clq}=P_{\clq}M_{\tilde{\Phi}\tilde{\Psi}}|_{\clq}
  =P_{\clq}M_{\tilde{\Psi}\tilde{\Phi}}|_{\clq},
 \]
and 
\[
 T_1\cong P_{\clq}M_{\tilde{\Phi}}|_{\clq},\
 T_2\cong P_{\clq}M_{\tilde{\Psi}}|_{\clq}.
\]
 \end{enumerate}
\end{Theorem}

We also have the following analogous result
for general $m$-hypercontractions.

\begin{Theorem}
\label{general factorization}
 Let $T$ be an $m$-hypercontraction on $\clh$. Then the following 
 are equivalent.
 \begin{enumerate}
  \item[\textup{$(i)$}]
 $T=T_1T_2$ for some $(T_1,T_2)\in\clf_m(\clh)$;
 \item [\textup{$(ii)$}]
 there exist a commuting pair of unitaries 
 $(W_1,W_2)$ on a Hilbert space $\clr$ with $W=W_1W_2$ and $\clb(\cld_{m,T})$-valued Schur functions
 \[
\tilde{\Phi}(z)=V^*(P+zP^{\perp})U^*V,\ \text{ and } 
\tilde{\Psi}(z)= V^*U(P^{\perp}+zP)V
\quad (z\in\D)
 \]
for some Hilbert space $\cle$, isometry $V:\cld_{m,T}\to \cle$,
unitary $U:\cle\to \cle$ and 
projection $P$ in $\clb(\cle)$
such that $\clq$ is a joint
 $(M_{\tilde{\Phi}}^*\oplus W_1^*, M_{\tilde{\Psi}}^*\oplus W_2^*)$-invariant subspace of $A^2_m(\cld_{m,T})\oplus \clr$,
 \[
  P_{\clq}(M_z\oplus W)|_{\clq}=
  P_{\clq}(M_{\tilde{\Phi}\tilde{\Psi}}\oplus W)|_{\clq}
  =P_{\clq}(M_{\tilde{\Psi}\tilde{\Phi}}\oplus W)|_{\clq},
 \]
and 
\[
 T_1\cong P_{\clq}(M_{\tilde{\Phi}}\oplus W_1)|_{\clq},\
 T_2\cong P_{\clq}(M_{\tilde{\Psi}}\oplus W_2)|_{\clq}.
\]
 \end{enumerate}
\end{Theorem}
An immediate consequence of the above results is a similar 
factorization result for subnormal operators. 
Recall that an operator is subnormal if it has a normal extension.
A well-known characterization of subnormal operator due to Agler 
is the following: a contraction $T$ on a Hilbert space $\clh$
is subnormal if and only if $T$ is an $m$-hypercontraction for 
all $m\in\N$ (see ~\cite{JA2}).  
We set 
\[\clf_{\infty}(\clh):=\bigcap_m\clf_m(\clh).\]
By the above characterization,
if $(T_1,T_2)\in \clf_{\infty}(\clh)$
then $T=T_1T_2$ is a subnormal operator. 
Thus $\clf_{\infty}(\clh)$ contains contractive factors 
of subnormal operators on $\clh$. A characterization of 
$\clf_{\infty}(\clh)$ is in order. 

\begin{Theorem}
\label{subnormal factorization}
 Let $T$ be a subnormal operator on $\clh$. Then the following 
 are equivalent.
 \begin{enumerate}
  \item[\textup{$(i)$}]
 $T=T_1T_2$ for some $(T_1,T_2)\in\clf_{\infty}(\clh)$;
 \item [\textup{$(ii)$}]
 for each $m\in\N$, there exist a commuting pair of unitaries 
 $(W_{1,m},W_{2,m})$ on a Hilbert space $\clr_m$ with $W_m=W_{1,m}W_{2,m}$
 and $\clb(\cld_{m,T})$-valued Schur functions
 \[
\tilde{\Phi}_m(z)=V_m^*(P_m+zP_m^{\perp})U_m^*V_m,\ \text{ and } 
\tilde{\Psi}_m(z)= V_m^*U_m(P_m^{\perp}+zP_m)V_m
\quad (z\in\D)
 \]
for some Hilbert space $\cle_m$, isometry $V_m:\cld_{m,T}\to \cle_m$,
unitary $U_m:\cle_m\to \cle_m$ and 
projection $P_m$ in $\clb(\cle_m)$
such that $\clq_m$ is a joint
 $(M_{\tilde{\Phi}_m}^*\oplus W_{1,m}^*, M_{\tilde{\Psi}_m}^*\oplus W_{2,m}^*)$-invariant subspace of $A^2_m(\cld_{m,T})\oplus \clr_m$,
 \[
  P_{\clq_m}(M_z\oplus W_m)|_{\clq_m}=
  P_{\clq_m}(M_{\tilde{\Phi}_m\tilde{\Psi}_m}\oplus W_m)|_{\clq_m}
  =P_{\clq_m}(M_{\tilde{\Psi}_m\tilde{\Phi}_m}\oplus W_m)|_{\clq_m},
 \]
and 
\[
 T_1\cong P_{\clq_m}(M_{\tilde{\Phi}_m}\oplus W_{1,m})|_{\clq_m},\
 T_2\cong P_{\clq_m}(M_{\tilde{\Psi}_m}\oplus W_{2,m})|_{\clq_m}.
\]
 \end{enumerate}
 \end{Theorem}

\section{Examples and concluding remark}
In this section, we find an example of a pair of commuting $2\times 2$ contractive 
matrices 
such that their product is a $2$-hypercontraction but the pair fails to belong 
in $\clf_2(\C^2)$. 

\textbf{Example}:
For a real number $0<r \le 1$, consider a \,$2 \times 2$\, matrix\, 
$T_r:= \begin{bmatrix}
   0 & r \cr 0 & 0
   \end{bmatrix}$. Then by a direct calculation, it can be checked that 
$T_r $ is a $2$-hypercontraction if and only if $r^2 \leq \frac{1}{2}$.
Also for strictly positive real numbers $a$ and $b$, consider the matrix $S= \begin{bmatrix}
   a & b \cr 0 & a 
\end{bmatrix}$. Then $S$ is an invertible matrix and $S$ commutes with $T_r$ for any $r$. Thus, for $r\le \frac{1}{\sqrt 2}$, $T_rS^{-1}$ and $S$ are factors of 
the $2$-hypercontraction $T_r$. On the other hand, again by a simple direct calculation, we have 
\begin{equation} 
\label{szego}
K^{-1}_1(T_r,T_r^*)-SK_1^{-1 }(T_r,T_r^*)S^*= \begin{bmatrix}
(1-r^2)(1-a^2) - b^2  &   -ab  \cr  -ab  &  1-a^2 
\end{bmatrix}.
\end{equation}
Also note that $S$ is a contraction if and only if $b\le 1-a^2$. So for the particular choice $r=\frac{1}{\sqrt{2}}$, $a=\frac{1}{\sqrt{2}}$
and $b=\frac{1}{2}$, we see that $T_r$ is a $2$-hypercontraction, $S$ and $T_rS^{-1}$ are contractions and 
\[
 K^{-1}_1(T_r,T_r^*)-SK_1^{-1 }(T_r,T_r^*)S^*= \begin{bmatrix}
0 &   -\frac{1}{2\sqrt{2}}  \cr  -\frac{1}{2\sqrt{2}}  &  \frac{1}{2} 
\end{bmatrix}
\]
is not a positive matrix. Therefore for such a particular choice, the contractions $T_rS^{-1}$ and $S$ are factors of the $2$-hypercontraction $T_r$ but 
$(T_rS^{-1}, S)\notin \clf_2(\C^2)$.

The above example shows that $\clf_m(\clh)$ does not contain all the 
contractive factors of $m$-hypercontractions on $\clh$ and the present article 
characterise a subclass of contractive factors of $m$-hypercontractions, namely 
$\clf_m(\clh)$. We conclude the paper with the following natural question: 
How to characterize all the factors of $m$-hypercontractions?

\vspace{0.1in} \noindent\textbf{Acknowledgement:}
The authors are grateful to Prof. Jaydeb Sarkar for his 
inspiration and valuable comments.
The research of the first name author is supported by the 
institute post doctoral fellowship of IIT Bombay.
The research of the second named author is supported by
DST-INSPIRE Faculty Fellowship No. DST/INSPIRE/04/2015/001094.

\bibliographystyle{amsplain}

\begin{thebibliography}{99}

\bibitem{JA1}
J. Agler, {\em The Arveson extension theorem and coanalytic models},
Integral Equations Operator Theory 5 (1982), 608-631.

\bibitem{JA2}
J. Agler, {\em Hypercontractions and subnormality},
J. Operator Theory 13 (1985), 203-217.

\bibitem{AM}
C. Ambrozie, {\em Commutative dilation theory}, 
Operator Theory (2015), 1093-1124.

\bibitem{AEM} 
C. Ambrozie, M. Englis and V. Muller,
{\em Operator tuples and analytic models
over general domains in $\C^n$},
J. Operator Theory 47 (2002), no. 2, 287-302.



\bibitem{BDHS}
S. Barik, B.K. Das, K. Haria and J. Sarkar, 
{\em Isometric dilation and von Neumann inequality 
for a class of tuples in the polydisc},
Tran. Amer. Math. Soc. 372 (2019), 1429-1450.

\bibitem{BDF1}
H. Bercovici, R.G. Douglas and C. Foias, {\em Canonical models for
bi-isometries}, A panorama of modern operator theory and related
topics, 177-205, Oper. Theory Adv. Appl., 218, Birkhauser/Springer
Basel AG, Basel, 2012.


\bibitem{BDF2}
H. Bercovici, R.G. Douglas and C. Foias, {\em Bi-isometries and
commutant lifting}, Characteristic functions, scattering functions
and transfer functions, 51-76, Oper. Theory Adv. Appl., 197,
Birkhauser Verlag, Basel, 2010.


\bibitem{BDF}
H. Bercovici, R.G. Douglas and C. Foias,
{\em On  the  classification  of  multi-isometries},
Acta Sci. Math.(Szeged) 72 (2006), 639-661.


\bibitem{BCL}
C.A. Berger, L.A. Coburn and A. Lebow,
{\em Representation and index theory for
$C^*$-algebras generated by commuting isometries}, 
J.Funct. Anal. 27 (1978) 51–99.

\bibitem{CV1}
R.E. Curto and F.H. Vasilescu,
{\em Standard operator models in the polydisc}, 
Indiana Univ. Math. J. 42 (1993), no. 3, 791-810. 

\bibitem{CV2}
R.E. Curto and F.H. Vasilescu,
{\em Standard operator models in the polydisc, II},
Indiana Univ. Math. J. 44 (1995), no. 3, 727-746.



\bibitem{DDS}
B.K. Das, R. Debnath and J. Sarkar,
{\em Commuting isometries and joint invariant 
subspace}, arXiv:1711.00769.

\bibitem{DS}
B.K. Das and J. Sarkar, {\em Ando dilations, von Neumann
inequality, and distinguished varieties}, J. Funct.
Anal. 272 (2017), 2114-2131.

\bibitem{DSS}
B.K. Das, S. Sarkar and J. Sarkar, {\em Factorization of Contraction}, 
Advances in Mathematics 322 (2017), 186-200.

\bibitem{Douglas}
R.G. Douglas, {\em On majorization,
factorization and range inclusion of operators on Hilbert space},
Proc. Amer. Math. Soc. 17 (1996), 413-415. 


\bibitem{GY}
K. Guo and R. Yang, {\em The core function of submodules over the
bidisk}, Indiana Univ. Math. J. 53 (2004), 205-222.

\bibitem{Yang}
W. He, Y. Qin and R. Yang, {\em Numerical invariants for commuting
isometric pairs}, Indiana Univ. Math. J. 64 (2015), 1--19.

\bibitem{MV}
V. Muller and F.-H. Vasilescu, {\em Standard models for some
commuting multioperators}, Proc. Amer. Math. Soc. 117 (1993),
979--989.

\bibitem{NF}
B. Sz.-Nagy and C. Foias, {\em Harmonic Analysis of Operators on
Hilbert Space}, North Holland, Amsterdam, 1970. 

\bibitem{O}
A. Olofsson, {\em Parts of Adjoint weighted shifts},
J. Operator Theory 74 (2015), no. 2, 249-280. 

\bibitem{Hari}
Haripada Sau, {\em Andô dilations for a pair of commuting contractions: two explicit constructions and functional models}, arxiv:1710.11368. 


\end{thebibliography}

\end{document}